\documentclass[11pt]{amsart}

\usepackage[a4paper,hmargin=3.5cm,vmargin=4cm]{geometry}
\usepackage{amsfonts,amssymb,amscd,amstext}
\usepackage{graphicx}
\usepackage{color}
\usepackage[dvips]{epsfig}

\usepackage[utf8]{inputenc}
\usepackage{hyperref}

\usepackage{verbatim}
\usepackage{fancyhdr}
\input xy
\xyoption{all}

\usepackage{times}
\usepackage{enumerate}
\usepackage{titlesec}
\usepackage{mathrsfs}

\numberwithin{equation}{section}
\numberwithin{figure}{section}

\titleformat{\section}
{\filcenter\bfseries\large} {\thesection{.}}{0.2cm}{}
\titleformat{\subsection}[runin]
{\bfseries} {\thesubsection{.}}{0.15cm}{}[.]
\titleformat{\subsubsection}[runin]
{\em}{\thesubsubsection{.}}{0.15cm}{}[.]

\usepackage[up,bf]{caption}

\newtheorem{theorem}{Theorem}[section]

\newtheorem{lemma}[theorem]{Lemma}
\newtheorem{corollary}[theorem]{Corollary}

\theoremstyle{definition}
\newtheorem{definition}[theorem]{Definition}
\newtheorem{remark}[theorem]{Remark}

\newtheorem{example}[theorem]{Example}

\newcommand\be{\mathbf{e}}

\newcommand\bv{\mathbf{v}}
\newcommand\bx{\mathbf{x}}

\def\bH{\mathbf{H}}

\def\bN{\mathbf{N}}

\newcommand\Cscr{\mathscr{C}}
\newcommand\Dscr{\mathscr{D}}

\newcommand\Gscr{\mathscr{G}}

\newcommand\C{\mathbb{C}}

\newcommand\CP{\mathbb{CP}}
\newcommand\D{\mathbb D}

\newcommand\N{\mathbb{N}}

\newcommand\R{\mathbb{R}}

\newcommand\T{\mathbb{T}}
\newcommand\Z{\mathbb{Z}}

\newcommand\cd{\overline{\mathbb D}}

\newcommand\igot{\mathfrak{i}}

\renewcommand\igot{\mathfrak{i}}

\newcommand\pgot{\mathfrak{p}}

\newcommand\ggot{\mathfrak{g}}

\newcommand\Igot{\mathfrak{I}}

\newcommand\E{\mathrm{e}}
\renewcommand\imath{\igot}

\newcommand\hra{\hookrightarrow}
\newcommand\lra{\longrightarrow}

\newcommand\wt{\widetilde}

\newcommand\di{\partial}

\renewcommand\div{\mathrm{div}}

\newcommand\Span{\mathrm{span}}
\newcommand\Laplace{\Delta}
\newcommand\nullq{{\mathbf A}}

\newcommand\Area{\mathrm{Area}}

\newcommand\TC{\mathrm{TC}}

\begin{document}

\fancyhead[LO]{Minimal surfaces for undergraduates}
\fancyhead[RE]{F.\ Forstneri\v c} 
\fancyhead[RO,LE]{\thepage}

\thispagestyle{empty}


\begin{center}
{\bf\LARGE Minimal surfaces for undergraduates}

\vspace*{0.5cm}

{\large\bf  Franc Forstneri{\v c}} 
\end{center}

\vspace*{0.5cm}

{\small
\noindent {\bf Abstract}\hspace*{0.1cm}
In this article we present an elementary introduction to the theory of 
minimal surfaces in Euclidean spaces $\R^n$ for $n\ge 3$ by using only elementary calculus 
of functions of several variables at the level of a typical second-year undergraduate 
analysis course for students of Mathematics at European universities. 
No prior knowledge of differential geometry is assumed.
}

\noindent{\bf Keywords}\hspace*{0.1cm} Minimal surface, Euler-Lagrange equation, Riemann surface 
\vspace*{0.1cm}

\noindent{\bf MSC (2020):}\hspace*{0.1cm}  Primary 53A10
%
%
%
%

\noindent {\bf Date: \rm 7 January 2021}


\section{Introduction}
\label{sec:intro}

Minimal surfaces are among the most beautiful and aesthetically pleasing geometric objects,
which are also of major physical importance.
These are surfaces in space which locally minimize the area, in the sense that any small enough piece
of the surface has the smallest area among all surfaces with the same boundary.
They appear naturally in the physical world. Laws of physics imply that a soap
film spanned by a given boundary curve assumes the shape of a minimal surface.
Furthermore, the most natural parameterization of a minimal surface by a smooth map, 
say from a planar domain, from the physical viewpoint is a conformal map, i.e., 
one which preserves angles between tangent vectors. The reason is that 
a conformal parameterization minimizes the total energy and makes 
the internal tension uniformly spread over the surfaces; see Remark \ref{rem:conformal}. 
This can be illustrated by the following experiment. Take a piece of cloth in the shape of a disc and made 
of elastic material. Now, stretch it as a curtain in the $3$-space with the boundary circle 
attached to a closed curve $C\subset\R^3$, but allowing it to slide freely along $C$.
It will assume the shape a minimal surface $S\subset \R^3$ with boundary $C$, 
and the position of points from the original disc inside this surface $S$ will yield a 
conformal parameterization of $S$ by the disc. 
It is however important that the boundary of the disc is not attached in a fixed way to 
the curve $C$, for otherwise a conformal parameterization cannot be achieved in general.
Such experiments, with soap bubbles as curtains, were conducted by 
Joseph Plateau in 1873, and they gave rise to the famous {\em Plateau problem} conjecturing
that any closed Jordan curve in $\R^3$ spans a minimal surface (in fact, a minimal disc). 
Plateau's conjecture was confirmed by Tibor Rad{\'o} \cite{Rado1930,Rado1930MZ} (1930) 
and Jesse Douglas \cite{Douglas1932} (1932).  

Minimal surfaces appear in a variety of applications. 
They are studied in any Riemannian manifold of dimension at least three, that is, 
a manifold with a smooth field of inner products on their tangent spaces. 
Holomorphic curves in complex Euclidean spaces $\C^n$ for $n>1$ are rather 
special examples minimal surfaces.

Unfortunately, this beautiful subject is not easily accessible to  
undergraduate students of Mathematics who already know the basic differential calculus 
of functions of several variables, but have not been exposed yet to Riemannian geometry. 
Presentations in standard texts (see 
\cite{Lawson1980,Osserman1986,BarbosaColares1986,Nitsche1989,ColdingMinicozzi1999,ColdingMinicozzi2011}, 
among others) either assume or develop several prerequisites from differential geometry 
before dealing with this topic, and the amount of necessary background material quickly becomes overwhelming. 

The immediate aim of writing this expository article was a desire to present rudiments of the theory of
minimal surfaces in a third year analysis course on differential equations and the 
calculus of variations. My intention was to discuss minimal surfaces 
as one of the examples of the Euler-Lagrange equation, but 
explaining also the geometric meaning of the resulting equation 
in terms of the vanishing of the mean curvature of the surface. Furthermore,
I wished to explore the role of conformal parameterization of the surface and the connection to
complex analysis via the Ennepper-Weierstrass formula. The latter opens the way 
to applying complex analytic methods in the theory of minimal surfaces, thereby
exposing a close connection between these two fields. These methods are highly efficient
as can be seen from the monographs \cite{Osserman1986,AlarconForstnericLopez2021book},
the AMS Memoir \cite{AlarconForstnericLopezMAMS}, and the recent survey
\cite{AlarconForstneric2019JAMS}.  All I could afford within the given syllabus were four or five lectures. 
This led me to develop an elementary presentation based on the following two principles.

\begin{itemize}
\item All calculations are done with the second order Taylor polynomials of relevant 
functions at a given base point, without an attempt to develop global formulas.
\item Metric notions such as length, area, and curvature coincide in any Euclidean coordinate system.
Explicitly, fixing a reference coordinate system on an affine space, any coordinate system 
obtained from the initial one by applying translations and orthogonal rotations
(the isometries of the Euclidean metric) is equally good.
\end{itemize}

By following these principles, I offer here an approach 
to the basics of the theory of minimal surfaces in Euclidean spaces $\R^n$ 
which is suitable for third year undergraduate 
students with no prior exposure to differential geometry.
The same principles apply in any Riemannian manifold; however, 
things become more involved since the derivatives of the metric enter
the picture, and the connection to complex analysis is lost in general.


\section{Graphs with minimal area}
\label{sec:graphs}

In this section we derive Lagrange's equation of minimal graphs, which is
one of the first examples in the calculus of variations for functions of more
than one variable. 

Let $D$ be a bounded domain  in the plane $\R^2$ with piecewise $\Cscr^1$ boundary $bD$. 
Given a function $f:\overline D\to \R$ of class $\Cscr^2$ on the closure of $D$, its graph 
\begin{equation}\label{eq:graphf}
	G_f = \bigl\{(x,y,z)\in \R^3: (x,y)\in \overline D,\ z=f(x,y)\bigr\}
\end{equation}
has area equal to
\begin{equation}\label{eq:Areaf}
	\Area(f) = \int_D \sqrt{1+f_x^2+f_y^2} \, dxdy = \int_D \sqrt{1+|\nabla f|^2} \, dxdy.
\end{equation}
Here, $f_x$ and $f_y$ denote the partial derivatives with respect to the indicated variables.

We are interested in finding functions $f$ for which the area is the smallest among all nearby
graphs over $\overline D$ having the same boundary values $f|_{bD}:bD\to \R$.
The first step is to understand when is $f$ a stationary point of the area functional.
To answer this question, we consider deformations of $f$ which are fixed on the boundary $bD$.
Pick a $\Cscr^1$ function $h:\overline D\to\R$ vanishing on $bD$ and
consider the function 
\[	
	\R\ni s\ \longmapsto\ \Area(f+sh)\in \R_+.
\]
Then, $f$ is a stationary point of the area functional \eqref{eq:Areaf} among all graphs 
over $\overline D$ with given boundary values if and only if 
\[ 
	\frac{d}{ds}\Big|_{s=0} \Area(f+sh)=0
\] 
holds for all such functions $h$. This expression equals
\begin{eqnarray}
	\frac{d}{ds}\Big|_{s=0} \Area(f+sh) &=& 
	\int_D \frac{d}{ds}\Big|_{s=0} \sqrt{1+(f_x+sh_x)^2+(f_y+sh_y)^2} \,dxdy 
	\cr \cr
	&=& \int_D \frac{f_xh_x+f_yh_y}{\sqrt{1+|\nabla f|^2}} \, dxdy. 
	\label{eq:calculation}
\end{eqnarray}
Integrating both summands by parts, we replace $h_x$ and $h_y$ by $h$
and put the respective derivative on the other term in the product. 
Since $h$ vanishes on $bD$, this gives 
\begin{equation}\label{eq:integralzero}
	\frac{d}{ds}\Big|_{s=0} \Area(f+sh)  = 
	- \int_D \left( \frac{\di}{\di x} \frac{f_x}{\sqrt{1+|\nabla f|^2}} +
	\frac{\di}{\di y} \frac{f_y}{\sqrt{1+|\nabla f|^2}} \right) h\, dxdy.
\end{equation}
This expression vanishes for all functions $h$ as above if and only if
\begin{equation}\label{eq:mingraph0}
	\frac{\di}{\di x} \frac{f_x}{\sqrt{1+|\nabla f|^2}} +
	\frac{\di}{\di y} \frac{f_y}{\sqrt{1+|\nabla f|^2}} = 0
\end{equation}
holds on $D$. Indeed, if this expression is nonzero at some point 
$(x_0,y_0)\in D$, say positive, then it is positive on a neighbourhood $U\subset D$ due
to continuity, and by choosing $h\ge 0$ to have support contained in $U$ 
and to satisfy $h(x_0,y_0)>0$, the integral on the right hand side of 
\eqref{eq:integralzero} is negative, so $f$ is not stationary. The equation  
\eqref{eq:mingraph0} can be written in the form
\[
	\div \left( \frac{\nabla f}{\sqrt{1+|\nabla f|^2}} \right)
	=\frac{(1+f_y^2) f_{xx}-2 f_x  f_y f_{xy}+(1+f_x^2) f_{yy}}{\left(1+|\nabla f|^2  \right)^{3/2}}=0,
\]
which is equivalent to 
\begin{equation}\label{eq:mingraph1}
	(1+f_y^2) f_{xx}-2 f_x  f_y f_{xy}+(1+f_x^2) f_{yy}=0.
\end{equation}

This calculation was made by J.-L.\ Lagrange in 1760. The second order 
quasilinear partial differential equation \eqref{eq:mingraph1} is known as the 
{\em equation of minimal graphs}, or the {\em minimal graph equation}. 
It is the {\em Euler-Lagrange equation}\index{Euler-Lagrange equation} 
for the area functional \eqref{eq:Areaf}. 

One may consider more general deformations of a given graph, with the same result.
Let $F(x,y,s)$ be a $\Cscr^2$ function on $\overline D \times (-\epsilon, +\epsilon)$ for some 
$\epsilon>0$ such that 
\begin{eqnarray*}\label{eq:generalvariation1}
	F(x,y,0) &=& f(x,y)\ \ \text{for all $(x,y)\in \overline D$, and} 
	\\
	\label{eq:generalvariation2}
	F(x,y,s) &=& f(x,y)\ \ \text{for all $(x,y)\in bD$ and $s\in (-\epsilon, +\epsilon)$}.
\end{eqnarray*}
Write 
\[
	F_s(x,y)=f(x,y)+ s h(x,y) + o(s),\quad h(x,y)=\frac{\di}{\di s}\Big|_{s=0} F_s(x,y).
\]
The same calculation which leads to \eqref{eq:calculation} then show that 
\[ 
	\frac{d}{ds}\Big|_{s=0} \Area(F_s)=\frac{d}{ds}\Big|_{s=0}\Area(f+sh).
\] 
This justifies the restriction to deformations which are linear in the parameter $s$ when 
deriving the minimal graph equation.

A natural question at this point is whether there exists a solution of 
the minimal graph equation \eqref{eq:mingraph1} with prescribed
continuous boundary values over $bD$, and if so, how many are there. 
This Dirichlet problem for the minimal graph equation was solved affirmatively for
any bounded convex domain $D\subset \R^2$ by T.\ Rad\'o \cite{Rado1930MZ} 
in 1930; an alternative proof was given by H.\ Jenkins and J.\ Serrin 
\cite{JenkinsSerrin1966} in 1966. The solution is unique and is an absolute area minimizer
among all surfaces with the given boundary. This is an important special case of
the Plateau problem mentioned in the introductory section.

%
%

\section{Curvature of surfaces}\label{sec:curvature}

In order to explain the geometric meaning of the minimal graph equation, 
we shall need the notion of principal curvatures and mean 
curvature of a surface in the Euclidean $3$-space $\R^3$. 
In coordinates $x=(x_1,x_2,\ldots,x_n)$ on $\R^n$ the Euclidean metric is given by 
\[
	ds^2=dx_1^2+dx_2^2+\cdots + dx_n^2.
\]
Its value on any pair of vectors $\xi=(\xi_1,\xi_2,\ldots,\xi_n)$
and $\eta=(\eta_1,\eta_2,\ldots,\eta_n)$ in $\R^n$ is $\sum_{i=1}^n\xi_i\eta_i = \xi\,\cdotp\eta$,
the standard Euclidean inner product of $\xi$ and $\eta$.

Let us observe the following, where the first item may be adopted as an axiom. 

\begin{itemize}
\item The curvature of an object (say a curve or a surface) is invariant under affine linear
maps $\R^n\to\R^n$ of the form $x\mapsto Ax+b$, where $b\in\R^n$ and $A\in O_n(\R)$ 
is an element of the orthogonal group on $\R^n$. Such maps are called {\em rigid}, and they are 
precisely the isometries of the Euclidean metric on $\R^n$. 
\smallskip
\item Every smooth curve $C$ can locally at any point $p\in C$ be represented as a
graph over its tangent line $T_p C$. The analogous property holds for smooth surfaces.
\end{itemize}

Hence, to explain the notion of curvature of a smooth plane curve $C\subset \R^2$ at a 
point $p\in C$, it suffices to apply a rigid change of coordinates in $\R^2$ taking $p$ to $(0,0)$ and the
tangent line $T_pC$ to the $x$-axis, so locally near $(0,0)$ the curve is
the graph $y=f(x)$ of a smooth function on an interval around $0\in\R$,  
with $f(0)=f'(0)=0$. (It will suffice to work with twice continuously differentiable functions.) 
The Taylor expansion of $f$ at $0$ is then 
\begin{equation}\label{eq:graphC}
	y= f(x) = \frac{1}{2} f''(0) x^2 + o(x^2).
\end{equation}
Let us find the circle which agrees with this graph to the second order at $(0,0)$.
Clearly, such a circle has centre on the $y$-axis, and hence is of the form 
\begin{equation}\label{eq:circle}
	x^2+(y-r)^2=r^2
\end{equation}
for some $r\in\R\setminus \{0\}$, unless $f''(0)=0$ when the $x$-axis $y=0$ 
(a circle of infinite radius) does the job. 
Solving the equation \eqref{eq:circle} on $y$ near $(0,0)$ gives
\[
	y=r-\sqrt{r^2-x^2}= r-r\sqrt{1-\frac{x^2}{r^2}} 
	= r-r \left(1 -  \frac{x^2}{2r^2} + o(x^2) \right)  
	= \frac{1}{2r} x^2+ o(x^2).
\] 
A comparison with \eqref{eq:graphC} shows that for $f''(0)\ne 0$ the number
\[
	r=1/f''(0) \in \R\setminus \{0\}
\]
is the unique number for which the circle \eqref{eq:circle} agrees with the curve \eqref{eq:graphC}
to the second order at $(0,0)$. This best fitting circle is called the {\em osculating circle}.
The number 
\begin{equation}\label{eq:kappa}
	\kappa=f''(0)=1/r 
\end{equation}
is the {\em signed curvature} of the curve \eqref{eq:graphC} at $(0,0)$, 
its absolute value $|\kappa|=|f''(0)|\ge 0$ is the {\em curvature},
and $|r|=1/|\kappa|=1/|f''(0)|$ is the {\em curvature radius}. If $f''(0)=0$ then
the curvature is zero and the curvature radius is $+\infty$.
The osculating circle lies in the upper half-plane $y\ge 0$ if $f''(0)>0$, and 
in the lower half-plane $y\le 0$ if $f''(0)<0$. 

Consider now a smooth surface $S\subset \R^3$.
Let $(x,y,z)$ be coordinates on $\R^3$.  Fix a point $p\in S$. A rigid change of coordinates gives
$p=(0,0,0)$ and $T_pS=\{z=0\}=\R^2\times \{0\}$. Then, $S$ is locally near the origin 
a graph of the form 
\begin{equation}\label{eq:graphS}
	z=f(x,y)= \frac{1}{2}\left( f_{xx}(0)x^2 + 2f_{xy}(0,0)xy + f_{yy}(0)y^2\right) + o(x^2+y^2).
\end{equation}
The symmetric matrix 
\begin{equation}\label{eq:Hessian}
	A=\left(\begin{matrix} f_{xx}(0,0) &  f_{xy}(0,0) \\ f_{xy}(0,0)  & f_{yy}(0,0) \end{matrix}\right)
\end{equation}
is called the {\em Hessian matrix} of $f$ at $(0,0)$.

Given a unit vector $v=(v_1,v_2)$ in the $(x,y)$-plane, 
let $\Sigma_v$ be the $2$-plane through $0\in \R^3$ spanned by $v$ and the $z$-axis.
The intersection $C_v := S\cap \Sigma_v$ is a curve in $S$ given by
\begin{equation}\label{eq:Cv}
	z= f(v_1t,v_2t) = \frac{1}{2} (Av\,\cdotp v) t^2  + o(t^2) 
\end{equation}
for $t\in \R$ near $0$. Since $|v|=1$, the parameters $(t,z)$ on the $\Sigma_v$ 
are Euclidean parameters, i.e., the Euclidean metric $ds^2$ on $\R^3$ restricted to the plane $\Sigma_v$
is given by $dt^2+dz^2$. From our discussion of curves and the formula \eqref{eq:kappa},  
we infer that the number 
\[
	\kappa_v = Av\,\cdotp v = f_{xx}(0)v_1^2 + 2f_{xy}(0,0)v_1v_2 + f_{yy}(0)v_2^2  
\]
is the signed curvature of the curve $C_v$ at the point $(0,0)$. 

On the unit circle $|v|^2=v_1^2+v_2^2=1$ the quadratic form $v\mapsto Av\,\cdotp v$
reaches its maximum $\kappa_1$ and minimum $\kappa_2$; these are the {\em principal curvatures} of 
the surface $S$ given by \eqref{eq:graphS} at $(0,0)$. Since the matrix $A$ is symmetric, 
$\kappa_1$ and $\kappa_2$ are its eigenvalues. The real numbers
\begin{equation}\label{eq:curvatures}
	H=\kappa_1+\kappa_2 = \mathrm{trace}\, A, \qquad K=\kappa_1 \kappa_2 =\det A
\end{equation}
are, respectively, the {\em mean curvature} and the {\em Gaussian curvature} of $S$ at $(0,0)$. 
(Sometimes the number $\frac12(\kappa_1+\kappa_2)$ is called the mean curvature.)

Note that the trace of $A$ \eqref{eq:Hessian} equals the Laplacian 
$\Laplace f (0,0)=f_{xx}(0,0)+f_{yy}(0,0)$. On the other hand, 
the trace of a matrix is the sum of its eigenvalues. This implies
\begin{equation}\label{eq:LaplaceisH}
	\Laplace f (0,0) =\kappa_1+\kappa_2 = H.
\end{equation}

Since the matrix $A$ is symmetric, the eigenvectors of $A$ corresponding to
the eigenvalues $\kappa_1$ and $\kappa_2$ are orthogonal. 
By an orthogonal rotation in the $(x,y)$-plane we can map these vectors to
$(1,0)$ and $(0,1)$, so the equation \eqref{eq:graphS} of the surface 
obtains the normal form
\begin{equation}\label{eq:normalS}
	z= f(x,y) = \frac{1}{2} \left( \kappa_1 x^2 + \kappa_2 y^2\right) + o(x^2+y^2).
\end{equation}

%
%
\section{Geometric interpretation of the minimal graph equation}\label{sec:mean}

We are now ready to prove the following theorem, due to Meusnier (1776), which
provides a geometric interpretation of the minimal graph equation  \eqref{eq:mingraph1}.

%
%
\begin{theorem}\label{th:Meusnier}
A $\Cscr^2$ function $f:D\to \R$ on a domain $D\subset\R^2$ satisfies the minimal graph equation 
\begin{equation}\label{eq:MGE}
	\Gscr(f):=(1+f_y^2) f_{xx}-2 f_x  f_y f_{xy}+(1+f_x^2) f_{yy}=0
\end{equation}

\noindent if and only if its graph $S=G_f$ \eqref{eq:graphf} has vanishing mean curvature at every point. 
\end{theorem}

\begin{proof}
Fix a point $p_0=(x_0,y_0)\in D$. Choose Euclidean coordinates on $\R^3$ 
which respect the $z$ direction (to keep the graph property) 
such that $p_0=(0,0)\in\R^2$, $f(0,0)=0$, and 
\[
	f(x,y)=ax+O(x^2+y^2), \quad a\ge 0.
\]  
If $p_0$ is a critical point of $f$, 
i.e., $f_x(p_0)=f_y(p_0) = 0$, then clearly $\Gscr(f)(p_0)= \Laplace f(p_0)= H$
where $H$ is the mean curvature of the graph $S=G_f$ at $p_0$ and the 
second equality holds by \eqref{eq:LaplaceisH}. 
Hence, the two conditions in the theorem are equivalent at such a point.
(It would be tempting to achieve $a=0$ by a rigid change of coordinates;
however, the orthogonal rotation on $\R^3$ which accomplishes this task 
does not preserve the graph condition.) Consider the orthonormal basis of $\R^3$ given by
\[
	\bv_1=\frac{1}{\sqrt{1+a^2}}(1,0,a),\quad \bv_2 = (0,1,0),\quad \bv_3=\frac{1}{\sqrt{1+a^2}}(-a,0,1).
\]
Then, $T_0 S=\Span\{\bv_1,\bv_2\}$ and $\bv_3$ is normal to $T_0S$. Let $(u,v,w)$ be the Euclidean 
coordinates associated to this basis, so the original coordinates are given by
\[
	(x,y,z)=u\bv_1+v\bv_2+w\bv_3.
\]
In coordinates $(u,v,w)$ the surface $S$ is given locally near the origin as a graph
\[
	w= g(u,v),\quad g(0,0)=0,\quad dg(0,0)=0. 
\]
By \eqref{eq:LaplaceisH}, the mean curvature of $S$ at $0$ equals $H=\Laplace g(0,0)$.

To complete the proof, we shall relate $g$ to $f$ and express $\Gscr(f)(0,0)$ in terms of $\Laplace g(0,0)$.
In the coordinates $(x,y,z)$ the surface $S$ is parameterized by 
\[
	x=\frac{1}{\sqrt{1+a^2}}\left(u-ag(u,v)\right),\quad y=v,\quad z=\frac{1}{\sqrt{1+a^2}}\left(au+g(u,v)\right).
\]
Since $S$ is also given by $z=f(x,y)$, we have the identity
\[
	 au+g(u,v) = \sqrt{1+a^2}\, \cdotp f\left(\frac{u-ag(u,v)}{\sqrt{1+a^2}},v\right).
\]
We now differentiate this identity twice on $u$ and $v$:
\begin{eqnarray*}
	a+g_u &=& f_x (1-ag_u) \\
	g_v     &=& f_x(-ag_v) + f_y \sqrt{1+a^2} \\
	g_{uu} &=& f_{xx}(1-ag_u)^2/\sqrt{1+a^2} + f_x (-ag_{uu}) \\
	g_{vv} &=& f_{xx}(-ag_v)^2/\sqrt{1+a^2} + f_x (-ag_{vv}) + f_{xy}(-2ag_v) + f_{yy}\sqrt{1+a^2}.	
\end{eqnarray*}
Evaluating these quantities at $(u,v)=(0,0)$ (which corresponds to $(x,y)=(0,0)$) and taking
into account that $f_x(0,0)=a,\ f_y(0,0)=0$, $g_u(0,0)=g_v(0,0)=0$ gives
\begin{eqnarray*}
	g_{uu}(0,0) &=& f_{xx}(0,0)/\sqrt{1+a^2} - a^2 g_{uu}(0,0), \\
	g_{vv}(0,0) &=&  -a^2 g_{vv}(0,0) + f_{yy}(0,0)\sqrt{1+a^2}, 
\end{eqnarray*}
and hence
\[
	f_{xx}(0,0) = {\sqrt{1+a^2}\,}^3 g_{uu}(0,0),\qquad f_{yy}(0,0)=\sqrt{1+a^2}\, g_{vv}(0,0).
\]
Therefore we get at $(0,0)$ that 
\begin{eqnarray*}
	\Gscr(f) &=& (1+f_y^2)f_{xx} - 2f_xf_y f_{xy} + (1+f_x^2)f_{yy} \\
		     &=& {\sqrt{1+a^2}\,}^3 g_{uu} +  (1+a^2) \sqrt{1+a^2} \, g_{vv} \\
		     &=& {\sqrt{1+a^2}\,}^3 \Laplace g.
\end{eqnarray*}
This shows that $\Gscr(f)(0,0)=0$ (the minimal surface equation holds at the origin) if and only if $\Laplace g(0,0)=H=0$
(the mean curvature of $S$ vanishes at $0\in S$). Since we considered an arbitrary point of $S=G_f$,
this proves the theorem.
\end{proof}

Theorem \ref{th:Meusnier} shows that a surface in $\R^3$ has vanishing mean curvature if and only if small pieces of
the surface are minimal graphs over affine planes. This motivates the following definition. 
We shall see that vanishing of the mean curvature is equivalent to the surface being
a stationary point of the area functional for general variations (see Theorem \ref{th:CMIMC}).

%
%
\begin{definition}\label{def:minimalsurface}
A smooth surface in $\R^3$ is a {\em minimal surface} if its mean curvature equals zero 
at every point: $\kappa_1+\kappa_2=0$.
\end{definition}

Every point in a minimal surface is a saddle point, and the surface 
is equally curved in both principal directions but in the opposite normal directions. 
Furthermore, the Gaussian curvature function
$K= \kappa_1 \kappa_2=-\kappa_1^2\le 0$ is nonpositive at every point.
The integral
\begin{equation}\label{eq:TC}
	\TC(S)=\int_S K\,\cdotp dA \in [-\infty,0]
\end{equation}
of the Gaussian curvature function with respect to the surface area on $S$ is called 
the {\em total Gaussian curvature}. This number equals zero if and only if $S$ is a 
piece of a plane.

The definition of a minimal surface extends to immersed surfaces 
since every immersion is locally near each point an embedding. More precisely, if
$M$ is a smooth surface and $F:M\to\R^3$ is a smooth immersion (i.e., 
its differential is injective), 
then every $p\in M$ has a neighbourhood $U\subset M$ such that $F(U)\subset \R^3$ 
is an embedded surface. We say that $F$ is a minimal immersion if every such surface $F(U)$ 
has vanishing mean curvature.

The global structure of a minimal surface can be complicated, 
and it need not be the image of a planar domain, although many of the oldest known examples 
such as the catenoid, the helicoid, Enneper's surface, Riemann's minimal examples, and many others
are parameterized by plane domains. In fact, any open (noncompact) surface, orientable
or nonorientable, can be realized as an immersed minimal surface in $\R^3$. Even more is true --- 
one can prescribe the conformal structure on such a surface; this notion will be explained 
in the following section. We refer to \cite[Chapter 3]{AlarconForstnericLopez2021book}
for more on this topic.

%
%
%
%
\section{Conformal parameterization of a minimal surface}\label{sec:conformal}

Let $(u,v)$ be coordinates on $\R^2$. An immersion $F:D\to\R^n$ from a domain $D\subset\R^2$
is said to be {\em conformal} if it preserves angles at every point. Explicitly, if $p\in D$ and
$\xi,\eta\in\R^2$ are nonzero vectors which determine an angle of size $0\le \theta\le \pi$, then 
the angle between the image vectors $dF_p (\xi), dF_p(\eta) \in \R^n$ also equals $\theta$.
It is elementary to see that an immersion is conformal if and only if it satisfies the following 
two conditions at every point $(u,v)\in D$:
\begin{equation}\label{eq:conformal}
	F_u\,\cdotp F_u = F_v\,\cdotp F_v, \qquad F_u\, \cdotp F_v =0.
\end{equation}
Here, the dot indicates the Eulidean scalar product.
See \cite[Lemma 1.8.4]{AlarconForstnericLopez2021book} for the details.

Recall that $\Delta$ denotes the Laplace operator. We shall need the following lemma.

%
%
\begin{lemma}\label{lem:Laplace-MC}
Let $D$ be a domain in $\R^2$. 
If $F=(F_1,F_2,\ldots,F_n): D\to \R^n$ is a conformal immersion of class $\Cscr^2$, then the vector 
$\Laplace F(p)=(\Laplace F_1(p),\ldots,\Laplace F_n(p))$ is orthogonal to the plane 
$dF_p(\R^2)\subset \R^n$ for every $p\in D$. Equivalently, 
\begin{equation}\label{eq:Laplaceorthogonal}
	\Laplace F \,\cdotp F_u=0,\qquad 	\Laplace F \,\cdotp F_v=0
\end{equation}
holds identically on $D$.
\end{lemma}

\begin{proof}
Differentiating the first identity in \eqref{eq:conformal} on $u$ and the second one on $v$ yields
\[
	 F_{uu} \,\cdotp  F_{u} =  F_{uv} \,\cdotp  F_{v} = - F_{vv} \,\cdotp  F_{u},
\]
whence $\Laplace F \, \cdotp  F_{u} = (F_{uu} + F_{vv})\, \cdotp  F_{u} = 0$. 
Likewise, differentiating the first identity on $v$ and
the second one on $u$ gives $\Laplace F \, \cdotp  F_{v} =0$. 
\end{proof}

Note that if $A$ is an $n\times n$ matrix and $\bx(u,v)=(x_1(u,v),\ldots,x_n(u,v)) \in\R^n$ is a 
smooth map $D\to\R^n$ from a plane domain $D\subset \R^2$, then  
\begin{equation}\label{eq:LaplaceR}
	\Laplace \left(A \,\cdotp \bx(u,v)\right) = A\,\cdotp \Laplace \bx(u,v).
\end{equation}
It follows that the property of the Laplacian $\Laplace F$ in Lemma \ref{lem:Laplace-MC} 
is invariant under rigid motions of $\R^n$ since they preserve angles.

We explained in the previous section how the notion of a minimal 
surface extends to immersions $F:D\to \R^3$ from plane domains.
The image $S=F(D)$ of such an immersion is called an {\em immersed minimal surface}. 
We will now show the following result. 

%
%
\begin{theorem}\label{th:CMI}
A conformal immersion $F=(x,y,z): D\to \R^3$ of class $\Cscr^2$ from a domain 
$D\subset\R^2$ parameterizes a surface with vanishing mean curvature function if and only if $F$ is harmonic:
\[ 
	\Laplace F=(\Laplace x,\Laplace y,\Laplace z)=0.
\] 
\end{theorem}

Furthemore, by Theorem \ref{th:CMIn} a conformal immersion $F:D\to\R^n$ is harmonic if and only
if it is a stationary point of the area functional, so the picture will be complete; see Theorem \ref{th:CMIMC}
which holds for any $n\ge 3$.

\begin{proof}
Fix a point $p_0\in D$; by a translation of coordinates we may assume that  $p_0=(0,0)\in \R^2$.
By \eqref{eq:LaplaceR} we are allowed to make the calculation in any Euclidean coordinate system on $\R^3$
obtained from the initial one by rigid motions. 
Since the differential $dF_{(0,0)}:\R^2\to\R^3$ is a conformal linear map and the mean curvature of a surface
is not affected by rigid motions of $\R^3$ either, we may assume that
\[
	F(0,0)=(0,0,0)\ \ \text{and}\ \ dF_{(0,0)}(\xi_1,\xi_2)=(\mu\xi_1,\mu\xi_2,0)\ \ 
	\text{for all $\xi=(\xi_1,\xi_2)\in\R^2$} 
\]	
for some $\mu>0$. Equivalently, at $(u,v)=(0,0)$ the following hold: 
\begin{equation}\label{eq:at0}
	 x_u=y_v=\mu>0,\quad x_v=y_u=0,\quad  z_u=z_v=0.
\end{equation}
Note that 
\begin{equation}\label{eq:mu}
	\mu = |F_u| = |F_v|.
\end{equation}
The implicit function theorem shows that there is a neighbourhood $U\subset D$ of the origin
such that the surface $S=F(U)$ is a graph $z=f(x,y)$ with $df_{(0,0)}=0$, so 
$f$ is of the form \eqref{eq:graphS}. 
By \eqref{eq:LaplaceisH} the mean curvature of $S$ at $0\in\R^3$ equals $H=\Laplace f(0,0)$.

Since the immersion $F$ is conformal, \eqref{eq:Laplaceorthogonal} shows that 
$\Laplace F$ is orthogonal to the $(x,y)$-plane $\R^2\times \{0\}$ at the origin, which means that 
\begin{equation}\label{eq:Laplaceat0}
	\Laplace x = \Laplace y = 0\ \ \text{at $(0,0)$}. 
\end{equation}
We now calculate $\Laplace z(0,0)$. Differentiation of $z(u,v)= f(x(u,v),y(u,v))$ gives
\[
	z_u=f_x x_u+f_y y_u,\qquad z_v=f_x x_v+f_y y_v,
\]
\[
	z_{uu} = \left(f_x x_u+f_y y_u\right)_u 
	= f_{xx} x_u^2 + f_{xy}x_uy_u + f_x x_{uu} + f_{yx}x_u y_u + f_{yy}y_u^2 + f_y y_{uu}. 
\]
At the point $(0,0)$, taking into account \eqref{eq:at0} and $f_x=f_y=0$
we get $z_{uu}=\mu^2  f_{xx}$. A similar calculation gives $z_{vv}=\mu^2  f_{yy}$ at $(0,0)$, 
so we conclude that
\begin{equation}\label{eq:Laplacezf}
	\Laplace z(0,0) = \mu^2 \Laplace f(0,0) = \mu^2 H,
\end{equation}
where $H$ is the mean curvature of the surface $F(U)$ at the origin. 
Denoting by $\bN=(0,0,1)$ the unit normal vector to $F(U)$ at $0\in\R^3$,
it follows from \eqref{eq:mu}, \eqref{eq:Laplaceat0} and \eqref{eq:Laplacezf} that 
\begin{equation}\label{eq:LaplaceH}
	\Laplace F = |F_u|^2 H \bN =  \frac12 |\nabla F|^2 H \bN
\end{equation}
holds at $(0,0)\in D$. In particular, $\Laplace F(0,0)=0$ if and only if 
$H=0$, i.e., the mean curvature of the image surface vanishes at $0$. 
Since this argument holds for an arbitrary point of $D$, this completes the proof.
\end{proof}

%
%
\begin{remark}[Mean curvature vector field]\label{rem:LaplaceH}
If $\bN$ is the unit normal vector field along the surface $F(D)\subset \R^3$ given by 
\[
	\bN = \frac{F_u \times F_v}{|F_u \times F_v|},
\]
the formula \eqref{eq:LaplaceH} holds at every point of $D$ as seen 
by rigid coordinate changes on $\R^3$. Note that the sign of the principal curvatures 
and of the mean curvature $H$ of a surface at a point depend on the choice 
of the normal vector $\bN$ at that point. Replacing $\bN$ by $-\bN$ also replaces $H$ by $-H$,
and hence the product $H \bN$ is independent of the choice of $\bN$. 
The vector field $H \bN$ is called the {\em mean curvature vector field} of $F$.
See \eqref{eq:MCV}, \eqref{eq:LaplaceHn} and Theorem \ref{th:CMIMC} for a generalization to minimal
surfaces in $\R^n$ for any $n\ge 3$.
\qed\end{remark}

A natural question appears: can we always replace a parameterization of an immersed
surface $F:D\to\R^n$ by a conformal one? The answer is affirmative locally near any given point of $D$, 
but globally the situation is more involved. Let us begin with the local picture.

An immersion $F:D\to\R^n$ determines on $D$ the Riemannian metric $g$, also called the 
{\em first fundamental form} of the immersed surface, by the formula
\[
	g= |F_u|^2 du^2 + (F_u\, \cdotp F_v) (du dv + dv du) + |F_v|^2 dv^2.
\]
This means that for any $p\in D$ and vectors $\xi=(\xi_1,\xi_2),\ \eta=(\eta_1,\eta_2)\in\R^2$ we have 
\begin{eqnarray*}
	g_p(\xi,\eta) &=& |F_u(p)|^2 \xi_1\eta_1 + F_u(p)\, \cdotp F_v(p) (\xi_1\eta_2+\xi_2\eta_1)
	+ |F_v(p)|^2 \xi_2\eta_2 \\
	&=& dF_p(\xi)\,\cdotp dF_p(\eta),
\end{eqnarray*}
so $F$ is an isometry from $D$ with the metric $g$ to $F(D)\subset \R^n$ with the Euclidean metric. 

The main point now is that any point $p\in D$ has a neighbourhood $U\subset D$ and 
coordinates $(\tilde u,\tilde v)$ on $U$ in which a given Riemannian metric $g$ assumes 
the simpler form
\[
	g= \lambda (d\tilde u^2 + d\tilde v^2)
\]
for some positive function $\lambda>0$. Any such coordinates $(\tilde u,\tilde v)$ are
called {\em isothermal coordinates} for the Riemannian metric $g$. 
Letting $\wt F=\wt F(\tilde u,\tilde v)$ be the immersion $U\to\R^n$
obtained from $F$ by expressing $(u,v)$ in terms of $(\tilde u,\tilde v)$, we get
\[
	|\wt F_{\tilde u}|^2 =  |\wt F_{\tilde v}|^2 = \lambda,\qquad 
		\wt F_{\tilde u}\,\cdotp \wt F_{\tilde v}=0,
\]
so $\wt F:U\to\R^n$ is a conformal immersion.

The existence of local isothermal coordinates was discovered by C.\ F.\ Gauss
for surfaces of revolution. The proof of the general case is beyond the reach of this article, 
and we refer to \cite[Section 1.8]{AlarconForstnericLopez2021book} where the optimal result 
(assuming only H\"older $\Cscr^\alpha$ regularity of the Riemannian metric for some $0<\alpha <1$)
is given, along with a brief history of the subject and references to the original articles and books. 

A reader who is only interested in the local picture 
may wish to skip the remainder of this section. For those brave enough, we now describe the global situation. 
Assume that $M$ is an abstract smooth surface
and $F:M\to\R^n$ is a smooth immersion. (A reader who is
not familiar with the basic  theory of manifolds may simply imagine that $M\subset \R^n$ 
is an embedded smooth surface and $F$ is the inclusion map.) By what was said above, 
we can parameterize a connected open neighbourhood $U\subset M$ of any given point $p\in M$ by
a smooth diffeomorphism $\phi :U'\to U$ from an open set $U'\subset \R^2$
such that $F\circ \phi:U'\to \R^n$ is a conformal embedding. 
If $\tilde \phi$ is another such local parametrization of a piece of $M$,
then the {\em transition map} $\phi^{-1}\circ \tilde \phi$ is a conformal diffeomorphism between
planar domains. (This map has nonempty domain only if the images of $\phi$ and $\phi'$
overlap, and its domain is the preimage of this overlap by $\tilde \phi$.) 
Identifying $\R^2$ with the complex plane $\C$, it is classical
that a conformal diffeomorphism between a pair of connected domains in $\C$ is either 
holomorphic or antiholomorphic, depending on whether it preserves or reverses the orientation.

A collection of local parameterizations of this kind whose images cover $M$
is said to be a {\em conformal atlas} on $M$, and a {\em complex atlas} 
if all transition maps are orientation preserving and therefore holomorphic. 
The latter case may be arranged if and only if the surface $M$ is orientable. 
The inverse $\phi^{-1}:U\to U'\subset\R^2$ of a local parameterization $\phi:U'\to U\subset M$ 
is called a {\em local chart} on $M$. A surface endowed with a conformal atlas is called 
a {\em conformal surface}, and one with a complex atlas is called a {\em Riemann surface}. 
For a conformal surface $M$, the notion of a conformal immersion $M\to \R^n$ 
is well defined by considering it in local coordinates from the conformal atlas.
There is an intrinsic notion of harmonic functions on a conformal surface, 
and the analogue of Theorem \ref{th:CMI} holds: 
{\em A conformal immersion $F:M\to\R^3$ from a conformal surface is a minimal immersion
if and only if $F$ is harmonic.}

The upshot is that the natural source surfaces to consider (when parameterizing minimal surfaces)  
are Riemann surfaces in the orientable case, and conformal surfaces in the nonorientable case. 
We refer to \cite[Sections 1.8, 1.9]{AlarconForstnericLopez2021book} for a more complete discussion.

How can one find a conformal parametrization in practice? 
In particular, if $F:\overline D\to\R^n$ is an immersion, is  there a diffeomorphism
$\phi:\overline D\to \overline D$ such that $F\circ \phi:\overline D\to \R^n$ is a {\em conformal} immersion?
The answer is affirmative if $\overline D$ is diffeomorphic to a closed disc.
In general, if $D$ is a bounded domain with piecewise smooth
boundary, we can find another domain $D'\subset\R^2$ of the same kind and a 
diffeomorphism $\phi:\overline{D'}\to \overline D$ such that $F\circ \phi: \overline {D'}\to \R^n$ 
is a conformal immersion. For example, if $D$ is an annulus
\[	
	A_r = \left\{(x,y)\in\R^2: 1/r^2 < x^2+y^2 < r^2\right\},\quad r>1,
\]
then we can take $D'=A_{r'}$ for a possibly different $r'>1$.
It is an elementary exercise in complex analysis that a pair of annuli 
$A_r$ and $A_{r'}$ as above are conformally equivalent (equivalently,  
biholomorphic) if and only if $r=r'$.

%
%
\section{Minimal surfaces in higher dimensional Euclidean spaces}\label{sec:higherdim}

What we have said so far extends to surfaces in Euclidean spaces $\R^n$
of dimension $n>3$. In this case there is no particular advantage in considering graphs
as we did in Section \ref{sec:graphs}, so we shall consider immersed surfaces
parameterized by plane domains. We start from the beginning and calculate the 
first variation of area at a conformal map \eqref{eq:gradient}, from which it follows
that $F$ is a stationary point of the area functional if and only if it is a harmonic map 
(see Theorem \ref{th:CMIn}). We then follow the approach developed in Section \ref{sec:conformal} 
for the case $n=3$ and show that a conformal immersion is harmonic if and only if its mean curvature
vector field vanishes identically; see Theorem \ref{th:CMIMC}. 

Denote the coordinates on $\R^n$ by $x=(x_1,x_2,\ldots,x_n)$.
Given a bounded domain $D$ in $\R^2_{(u,v)}$ with piecewise
$\Cscr^1$ boundary and a $\Cscr^2$ immersion $F=(F_1,\ldots,F_{n}):\overline D\to\R^n$, 
the area of the image surface $F(\overline D)\subset \R^n$ (counting multiplicities) is given by
\begin{equation}\label{eq:Areafn}
	\Area(F) = \int_D \sqrt{|F_u|^2 |F_v|^2 - |F_u\,\cdotp  F_v|^2} \, dudv.
\end{equation}
As before, $F_u$ and $F_v$ denote the partial derivatives of $F$ with respect to the indicated variables,
which are $\Cscr^1$ functions with values in $\R^n$. The gradient $\nabla F =(F_u,F_v)$ is represented by 
an $n\times 2$ matrix and $|\nabla F|^2=|F_u|^2+ |F_v|^2$.

If $F$ is a conformal immersion (see \eqref{eq:conformal}), then the formula \eqref{eq:Areafn} simplifies to 
\begin{equation}\label{eq:Areaconformal}
	\Area(F) = \int_D |F_u|^2 \, dudv = \frac{1}{2} \int_D \left(|F_u|^2+ |F_v|^2\right) \, dudv
	= \frac{1}{2} \Dscr(F) ,
\end{equation}
where 
\begin{equation}\label{eq:energy}
	\Dscr(F) = \int_D |\nabla F|^2 \, dudv	
\end{equation}
is the {\em Dirichlet (energy) integral} of the map $F$. 
(See Remark \ref{rem:conformal} for more on this.)

We now calculate the first variation of $\Area(F)$, assuming that $F$ is conformal.
Let $h:\overline D\to\R^n$ be a $\Cscr^1$ map vanishing on $bD$.  Consider the expression
under the integral \eqref{eq:Areafn} for the map $F_s=F+sh$, $s\in \R$. (Note that $F_s$ is an immersion
for $s$ close to $0$.) Taking into account that $|F_u|=|F_v|$ and $F_u\,\cdotp F_v=0$
(see \eqref{eq:conformal}) we obtain
\[
	|F_u+sh_u|^2 \,\cdotp |F_v+sh_v|^2 = |F_u|^4 + 
	2s \left(F_u\,\cdotp h_u + F_v\,\cdotp h_v\right) |F_u|^2 + O(s^2),
\]
\[
	|(F_u+s h_u)\,\cdotp  (F_v+sh_v)|^2 = O(s^2).
\]
It follows that 
\begin{multline*}
	\frac{d}{ds}\Big|_{s=0} \bigl(|F_u+sh_u|^2 |F_v+sh_v|^2 - |(F_u+sh_u)\,\cdotp (F_v+sh_v)|^2 \bigr)  \cr
	= 2 \left(F_u\,\cdotp h_u + F_v\,\cdotp h_v\right) |F_u|^2 
\end{multline*}
and therefore
\begin{eqnarray}
	\frac{d}{ds}\Big|_{s=0} \Area(F+sh) 
	&=& \int_D \left( F_u\,\cdotp h_u + F_v\,\cdotp h_v\right)  dudv \cr\cr
	&=& -\int_D \Laplace F \cdotp h \, dudv.
	\label{eq:gradient}
\end{eqnarray}
(We integrated by parts and used that $h|_{bD}=0$.
Note that the factor $|F_u|^2$ also appears in the denominator when differentiating the 
expression for $\Area(F+sh)$ under the integral at $s=0$, so these terms cancel.)
Clearly, this expression vanishes for all maps $h$ vanishing on $bD$ if and only if $\Laplace F=0$,
which gives the following result. 

\begin{theorem}\label{th:CMIn}
A conformal immersion $F:D\to \R^n$ $(n\ge 3)$ of class $\Cscr^2$ is a stationary
point of the area functional \eqref{eq:Areafn} if and only if $F$ is harmonic.
\end{theorem} 

A conformal immersion satisfying the equivalent conditions in the above theorem is said to be
a {\em conformal minimal immersion}; it parameterizes an immersed minimal surface in $\R^n$.
Theorem \ref{th:CMIMC} proved below shows that these two conditions are further equivalent
to vanishing of the mean curvature vector field, so the picture will be complete.

%
%
\begin{remark}[Gradient of the area functional] \label{rem:gradient}
The identity \eqref{eq:gradient} says more. The quantity
\begin{equation}\label{eq:gLaplacian}
	\frac{2}{|\nabla F|^2}\Laplace F=\Laplace_g F
\end{equation} 
is the {\em intrinsic Laplacian} of $F=(F_1,\ldots,F_n)$ with respect to
the $F$-induced Riemannian metric $F^*ds^2=|dF_1|^2+\cdots +|dF_n|^2$ on $D$, 
and $dA=\frac12 |\nabla F|^2 dudv$ is the associated area measure on the image surface
as seen from \eqref{eq:Areaconformal}. Hence, \eqref{eq:gradient} can be written in the form
\[
	\frac{d}{ds}\Big|_{s=0} \Area(F+sh) = -\int_D \Laplace_g F \,\cdotp h \, dA.
\]
This is the {\em first variational formula for the area} of an immersed surface.
It may be interpreted by saying that the vector field $\Laplace_g F$ 
along a conformal immersion $F:D\to\R^n$ is the negative gradient of the area functional
at $F$. In other words, deforming the surface in the direction of $\Laplace_g F$,
keeping it fixed on the boundary of a domain in $D$, leads to the fastest decrease of the area
of that piece of the surface. As we shall presently see, $\Laplace_g F$ equals 
the {\em mean curvature vector field} of $F$; cf.\  \eqref{eq:LaplacegH}.
\qed\end{remark}

The subsequent geometric analysis is much like what we have done in
the case $n=3$ in Section \ref{sec:conformal}. Assume that $F:D\to \R^n$ is a conformal immersion
of class $\Cscr^2$. Fix a point $p_0=(u_0,v_0)\in D$ and assume by translations that  
$p_0\in\R^2$ and $F(0,0)=0\in\R^n$. Denote by $\be_1,\ldots,\be_n$ the standard basis 
vectors of $\R^n$. Postcomposing $F$ by an element of the orthogonal group $O_n(\R)$
we may assume that
\[
	F(0,0)=0\in\R^n,\quad F_u(0,0)=\mu \be_1,\quad F_v(0,0)=\mu \be_2
\]
where $\mu = |F_u(0,0)|=|F_v(0,0)|$. By the implicit function theorem there is a neighbourhood 
$U\subset D$ of $(0,0)$ such that the embedded surface $S=F(U)\subset \R^n$ is a graph 
over the domain $V=(F_1,F_2)(U)\subset \R^2$ given by
\[
	S=\big\{(x_1,x_2,\ldots,x_n):(x_1,x_2)\in V,\ x_i=f_i(x_1,x_2)\ \text{for}\ i=3,\ldots,n\big\},
\]
where each $f_i:V\to\R$ is a $\Cscr^2$ function of the form 
\[ 
	f_i(x_1,x_2) = \frac{1}{2}\left( \alpha_i x_1^2 + 2\beta_i x_1x_2 + \gamma_i x_2^2\right) 
	+ o(x_1^2+x_2^2).
\] 
The matrix 
$
	A_i=\left(\begin{matrix} \alpha_i & \beta_i\cr \beta_i &  \gamma_i \end{matrix}\right)
$
is the Hessian matrix of $f_i$ at $(0,0)$; see \eqref{eq:Hessian}. 
The discussion in Section \ref{sec:curvature} 
shows that the eigenvalues $\kappa^i_1,\ \kappa^i_2$ of $A_i$ 
are the principal curvatures at $0$ of the surface $S_i$ in the $3$-dimensional
space $\R^3_{x_1,x_2,x_i}$ obtained by orthogonally projecting $S$ to it (i.e., 
neglecting the components $x_j$ for $j\in \{3,\ldots,n\}\setminus \{i\}$).
The numbers
\[
	H_i= \kappa^i_1 + \kappa^i_2=\Laplace f_i(0,0),\qquad K_i=\kappa^i_1 \kappa^i_2
\]
are, respectively, the mean curvature and the Gaussian curvature of $S$ at $0$ in 
the direction of the normal vector $\be_i$ for $i\in \{3,\ldots,n\}$. The vector
\begin{equation}\label{eq:MCV}
	\bH= \sum_{i=3}^n H_i \be_i = \sum_{i=3}^n  \Laplace f_i(0,0) \be_i 
\end{equation}
is the {\em mean curvature vector} of the surface $S$ at $0\in S$. 
Since $F$ is conformal, the vector $\Laplace F$ is orthogonal to $F$ at every point
by Lemma \ref{lem:Laplace-MC}, which implies that 
\[
	\Laplace F_1(0,0)=\Laplace F_2(0,0)=0.
\]
The same calculation as in the case $n=3$ (see in particular \eqref{eq:Laplacezf}) 
then shows that 
\[
	\Laplace F_i(0,0) = \mu^2 \Laplace f_i(0,0) = \mu^2 H_i,\quad\ i=3,\ldots,n.
\]
Since $\mu^2=|F_u(0,0)|^2=|F_v(0,0)|^2=\frac12 |\nabla F(0,0)|^2$, the upshot is that 
\[
	\Laplace F(0,0)= \frac12 |\nabla F(0,0)|^2 \, \bH.
\]
Applying rigid motions of $\R^n$ as we did in Section \ref{sec:conformal} for surfaces in $\R^3$, 
we see that the same holds true at every point of $D$. Together with Theorem \ref{th:CMIn}
this proves the following key result in the theory of minimal surfaces in Euclidean spaces.

%
%
\begin{theorem}\label{th:CMIMC} 
A conformal immersion $F:D\to\R^n$ of class $\Cscr^2(D)$ satisfies 
\begin{equation}\label{eq:LaplaceHn}
	\Laplace F = \frac12 |\nabla F|^2\, \bH,
\end{equation}
where $\bH$ is the mean curvature vector field of the image surface. In particular,
$F$ is a stationary point of the area functional \eqref{eq:Areafn} 
if and only if the mean curvature vector field $\bH$ 
of the immersed surface $F(D)\subset \R^n$ vanishes identically.
\end{theorem}

In terms of the intrinsic metric Laplacian \eqref{eq:gLaplacian} of the immersion $F$, the formula
in the above theorem assumes the simpler form
\begin{equation}\label{eq:LaplacegH}
	\Laplace_g F = \bH. 
\end{equation}

%
%
\begin{remark}\label{rem:conformal}
As already mentioned in the introduction, 
conformal parameterizations of minimal surfaces are the most natural ones from the 
physical perspective. Given a $\Cscr^1$ immersion $F:\overline D\to\R^n$,
let $\Dscr(F)$ be its Dirichlet energy integral \eqref{eq:energy}. 
Note that for any two vectors $x,y\in\R^n$ we have 
\[
	|x|^2|y|^2 - |x\,\cdotp y|^2 \le |x|^2|y|^2  \le \frac14 \left(|x|^2+|y|^2\right)^2
\]
with equality if and only if $x\,\cdotp y=0$ (for the first inequality) and $|x|=|y|$ (for the second one).
This simple observation shows that 
\[
	\Area(F)\le \frac{1}{2} \Dscr(F)\ \ \text{with equality if and only if $F$ is conformal}.
\] 
Assume now that $\overline D$ is the closed unit disc $\cd=\{(u,v)\in\R^2:u^2+v^2\le 1\}$. 
Consider maps $F:\cd\to \R^n$ of class $\Cscr^2(\cd)$ 
whose restriction to the boundary circle $\T=b\D$ is an injective
parameterization $F:\T\to \Gamma\subset \R^n$ of a given smooth oriented Jordan curve.
A map in this class minimizing the Dirichlet integral $\Dscr(F)$ also minimizes the area
and provides a conformally parameterized minimal surface with boundary $\Gamma$
(see Lawson \cite[Sect.\ II.1]{Lawson1980}). In other words, {\em conformal parameterization gives a
least energy spreading of the surface over a geometric configuration of least area in $\R^n$.}
This is in analogy to minimization of the energy integral of curves in a Riemannian manifold
which yields geodesics (curves of minimal length) parameterized by constant multiples of arc length.
\qed\end{remark}

%
%
\section{A complex analytic viewpoint on minimal surfaces}\label{eq:complex}

In this section we explain the Enneper-Weierstrass formula which provides a 
connection between holomorphic maps $D\to\C^n$ with special properties
from domains $D\subset \C$ and conformal minimal immersions $D\to\R^n$ for $n\ge 3$.
The same connection holds more generally for maps from any open Riemann surface
instead of a plane domain.

Let $z=x+\imath y$ be a complex coordinate on $\C$. Let us recall the following basic 
operators of complex analysis:
\[
	\frac{\di}{\di z}= \frac12 \left(\frac{\di}{\di x} - \imath \frac{\di}{\di y}\right), \qquad
	\frac{\di}{\di \bar z}= \frac12 \left(\frac{\di}{\di x} + \imath \frac{\di}{\di y}\right).
\]
The kernel of $\frac{\di}{\di \bar z}$ consists of holomorphic function, and the 
kernel of $\frac{\di}{\di z}$ consists of antiholomorphic functions. The differential 
of a function $F$ can be written in the form
\[
	dF = \frac{\di F}{\di x}dx +  \frac{\di F}{\di y}dy 
	= \frac{\di F}{\di z}dz + \frac{\di F}{\di \bar z}d\bar z
\]
where 
\[
	dz=dx+\imath dy,\qquad d\bar z=dx-\imath dy.
\]
Note that $\frac{\di F}{\di z}dz$ is the $\C$-linear part and $\frac{\di F}{\di \bar z}d\bar z$ is the 
$\C$-antilinear part of the differential $dF$. In terms of these operators, the Laplacian equals
\[
	\Laplace =  \frac{\di^2}{\di x^2}+\frac{\di^2}{\di y^2} 
	= 4 \frac{\di}{\di \bar z} \frac{\di}{\di z} =  4 \frac{\di}{\di z}\frac{\di}{\di \bar z}.
\]
Hence, a function $F:D\to\R$ is harmonic if and only if the function $\di F/\di z$ is holomorphic.
It follows that a smooth map $F=(F_1,F_2,\ldots,F_n):D\to \R^n$ is a harmonic immersion if and only if
the map $f=(f_1,f_2,\dots, f_n):D\to \C^n$ with components $f_j=\di F_j/\di z$ for $j=1,2,\ldots,n$
is holomorphic and the component functions $f_j$ have no common zero. (At a common zero
of these functions, $F$ would fail to be an immersion.) Furthermore, conformality of $F$
is equivalent to the following nullity condition:
\begin{equation}\label{eq:nullity}
	f_1^2 + f_2^2 +\cdots+ f_n^2 = 0\ \ \ \text{on $D$}.
\end{equation}
Indeed, we have that 
\[
	4f_j^2 =  \left( F_{j,x} -\imath F_{j,y} \right)^2 = (F_{j,x})^2 - (F_{j,y})^2 - 2\imath F_{j,x} F_{j,y}.
\]
Summation over $j=1,\ldots,n$ gives
\[
	4 \sum_{j=1}^n f_j^2  = |F_x|^2-|F_y|^2 - 2\imath F_x\,\cdotp F_y.
\]
Comparing with the conformality conditions \eqref{eq:conformal} proves the claim. 

Since we know by Theorem \ref{th:CMIn} that a conformal immersion is harmonic 
if and only it parameterizes a minimal surface, this gives the following result.

%
%
\begin{theorem}[The Enneper-Weierstrass representation theorem]\label{th:EW}
Let $D$ be a connected domain in $\C$.
A map $F=(F_1,F_2,\ldots,F_n):D\to\R^n$ of class $\Cscr^2$ is a conformal minimal immersion if and only 
if the map $f=(f_1,f_2,\ldots,f_n)=\di F/\di z : D\to \C^n \setminus \{0\}$ is holomorphic and satisfies
the nullity condition \eqref{eq:nullity}.

Conversely, a holomorphic map $f=(f_1,f_2,\ldots,f_n):D\to \C^n\setminus \{0\}$ 
satisfying the nullity condition \eqref{eq:nullity} and the period vanishing conditions
\begin{equation}\label{eq:Rperiod}
	\Re \oint_C f\, dz=0\ \ \text{for every closed curve $C\subset D$}
\end{equation}
determines a conformal minimal immersion $F:D\to\R^n$ given by
\begin{equation}\label{eq:primitive}
	F(z)= c + 2\Re \int_{z_0}^z f(\zeta)\, d\zeta,\quad z\in D
\end{equation}
for any base point $z_0\in D$ and vector $c=(c_1,c_2,\ldots,c_n)\in\R^n$. 
\end{theorem}

The real period vanishing conditions \eqref{eq:Rperiod} guarantee that the integral
in \eqref{eq:primitive} is well defined, in the sense that it is independent of the choice
of a path of integration in $D$ from the initial point $z_0\in D$ to the terminal point $z\in D$.
The imaginary components 
\begin{equation}\label{eq:flux}
	\Im \oint_C f\, dz  = \pgot(C) \in \R^n
\end{equation}
of the periods $\oint_C f\, dz$ define the {\em flux homomorphism} $\pgot:H_1(D,\Z)\to\R^n$. 
Indeed, by Green's formula (the planar version of Stokes's theorem)
the period $\oint_C f\, dz$ only depends on the homology
class $[C]\in H_1(D,\Z)$ of a closed path $C\subset D$. 

%
%
\begin{remark}[The first homology group] 
If $D$ is a domain in $\R^2\cong \C$ then its first homology group $H_1(D,\Z)$ 
is a free abelian group $H_1(D,\Z)\cong \Z^\ell$ 
$(\ell\in \{0,1,2,\ldots\}\cup\{\infty\})$ with finitely or countably many generators.
If $D$ is bounded, connected, and its boundary $bD$ consists of $l_1$ Jordan curves 
$\Gamma_1,\ldots, \Gamma_{l_1}$ and $l_2$ isolated points (punctures) $p_1,\ldots, p_{l_2}$, 
then the group $H_1(D,\Z)$ has
$\ell = l_1+l_2-1$ generators which are represented by loops in $D$ based at any given 
point $p_0\in D$ which surround each of the $\ell$ holes of $D$. By a {\em hole},
we mean a bounded (hence compact) connected component of the complement $\C\setminus D$.
Indeed, if $\Gamma_1$ is the outer boundary curve of $D$, then every other boundary curve 
$\Gamma_2,\ldots, \Gamma_{l_1}$ of $D$ also bounds a hole of $D$, and each of the
points $p_1,\ldots,p_{l_2}$ is a hole.  Every hole contributes one generator to $H_1(D,\Z)$. 
The same loops then generate the fundamental group $\pi_1(D,p_0)$ 
as a free nonabelian group. (In general, the first homology group $H_1(D,\Z)$ is the abelianisation
of the fundamental group $\pi_1(D,p_0)$.) A similar description of $H_1(D,\Z)$ holds for every surface, 
except that its genus enters the picture as well; 
see \cite[Sect.\ 1.4]{AlarconForstnericLopez2021book}. 
For basics on homology and cohomology, see J.\ P.\ May \cite{May1999}. 
\qed\end{remark}

It is clear from Theorem \ref{th:EW} that the following quadric complex hypersurface in $\C^n$ 
plays a special role in the theory of minimal surfaces in Euclidean spaces: 
\begin{equation}\label{eq:nullquadric}
	\nullq = \nullq^{n-1}
	= \bigl\{(z_1,\ldots,z_n) \in\C^n : z_1^2+z_2^2 + \cdots + z_n^2 =0\bigr\}.
\end{equation}
This is called the {\em null quadric} in $\C^n$.
Removing the origin, we get the {\em punctured null quadric}
$\nullq_*=\nullq\setminus\{0\}$. Note that $\nullq$ is a complex cone with
the only singular point at $0$, since the differential of the defining function is
nonzero elsewhere. Theorem \ref{th:EW} says that we get all conformal
minimal surfaces in $\R^n$ which are parameterized by a domain $D\subset \C$
as integrals (primitives) of holomorphic maps $f:D\to \nullq_* \subset \C^n$
satisfying the real period vanishing conditions \eqref{eq:Rperiod}.

Let us now look at a family of immersed holomorphic curves in $\C^n$ which are close relatives
of conformally immersed minimal surfaces in $\R^n$.

%
%
\begin{definition} \label{def:nullcurve}
A holomorphic immersion $F=(F_1,\ldots,F_n):D\to \C^n$ for $n\ge 3$ from a domain $D\subset \C$
satisfying the nullity condition
\[
	(F'_1)^2 + (F'_2)^2 + \cdots + (F'_n)^2 =0
\]
is a {\em holomorphic null curve}\index{holomorphic null curve} in $\C^n$.
\end{definition}

Hence, the derivative $f=F':D\to \C^n_*$ of a holomorphic null curve
assumes values in the punctured null quadric $\nullq^{n-1}_*$ \eqref{eq:nullquadric}.
For any closed curve $C\subset D$ we clearly have $\oint_C fdz= \oint_C dF=0$. 
Conversely, a holomorphic map $f:D\to\nullq^{n-1}_*$ satisfying the 
complex period vanishing conditions 
\begin{equation}\label{eq:Cperiods}
	\oint_C fdz =0\quad \text{for every closed curve $C\subset D$}
\end{equation}
integrates to a holomorphic null curve
\begin{equation}\label{eq:nullcurve}
	F(z)=c + \int_{z_0}^z f(\zeta)d\zeta,\qquad z\in D,
\end{equation}
where $z_0\in D$ is any given base point and $c\in\C^n$. Indeed, the 
conditions \eqref{eq:Cperiods} guarantee that the integral in \eqref{eq:nullcurve} 
is independent of the choice of a path of integration.

If $f:D\to\C^n$ is a holomorphic map then by Green's formula the period 
$\oint_C f\, dz\in\C^n$ only depends on the homology class $[C]\in H_1(D,\Z)$ of $C$. 
In particular, these periods vanish if the domain has trivial homology group
$H_1(D,\Z)=0$, which is equivalent to $D$ being simply connected.
According to the Riemann mapping theorem, there are precisely two such 
domains up to biholomorphisms: the plane $\C$ and the disc $\D=\{|z|<1\}$.

%
%
\begin{corollary}\label{cor:EW}
If $D$ is a simply connected domain in $\C$ then every holomorphic map 
$f:D\to\nullq_*\subset \C^n$ determines a holomorphic null curve by the formula \eqref{eq:nullcurve}.
\end{corollary}

%
%
If $Z=X+\imath Y:D\to \C^n$ is a holomorphic null curve then its real part
$X=\Re Z:D\to\R^n$ and its imaginary part $Y=\Im Z:D\to\R^n$ are conformal minimal surfaces.
Indeed, denoting the complex variable in $\C$ by $z=x+\imath y$ and taking into account 
the Cauchy-Riemann equations which are satisfied by a holomorphic map $Z$, we have that 
\[
	f=Z' = Z_x = X_x+\imath Y_x= X_x-\imath X_y = \frac12 \frac{\di X}{\di z}.
\]
Since $f=Z':D\to\nullq^{n-1}_*$ satisfies the nullity condition \eqref{eq:nullity}, 
$X$ is a conformal minimal immersion.
In the same way we find that $f=Z'=Y_y+\imath Y_x=\frac{\imath}{2} Y_z$,
so $Y$ is a conformal minimal immersion. Being harmonic conjugates 
of each other, $X$ and $Y$ are called {\em conjugate minimal surfaces}.
Conformal minimal surfaces in the $1$-parameter family
\[
	X^{\,t}=\Re(\E^{\imath \, t}Z) : D\lra \R^n,\quad\  t\in\R
\]
are called {\em associated minimal surfaces} of the holomorphic null curve $Z$. 

Conversely, if $X:D\to\R^n$ is a conformal minimal surface and 
$f=\frac12 \frac{\di X}{\di z}:D\to \nullq^{n-1}$ satisfies the period vanishing conditions 
\eqref{eq:Cperiods}, then $f$ integrates to a holomorphic null curve $Z:D\to\C^n$ \eqref{eq:nullcurve}
with $\Re Z=X$. Recall that the imaginary parts of the periods \eqref{eq:nullcurve} determine 
the {\em flux} of the minimal surface $X$; hence, $X$ is the real part of a holomorphic null curve 
if and only if it has vanishing flux. Note that the periods \eqref{eq:Cperiods} always vanish 
on a simply connected domain $D\subset \C$, and hence every conformal minimal immersion 
$D\to\R^n$ from such a domain is the real part of a holomorphic null curve $D\to\C^n$.

%
%
\begin{example}[Helicatenoid] \label{ex:helicatenoid}
Consider the holomorphic immersion $Z:\C\to\C^3$ given by
\begin{equation}\label{eq:helicatenoid}
	Z(z) = (\cos z,\sin z,-\imath z)\in \C^3, \quad\  z=x+\imath y\in\C.
\end{equation}
We have that 
\[
	f(z)=Z'(z)=(-\sin z,\cos z,-\imath),\quad \sin^2 z + \cos^2 z + (-\imath)^2=0.
\]
Hence, $Z$ is a holomorphic null curve. Consider its associated minimal surfaces in $\R^3$:
\begin{equation}\label{eq:helicatenoidassociated}
    X^{\,t}(z) = \Re\left( \E^{\imath \, t} Z(z) \right)
     = \cos t \left(
     \begin{matrix} \cos x \,\cdotp \cosh y \cr \sin x \,\cdotp \cosh y \cr y \end{matrix}\right)
      + \sin t \left(
      \begin{matrix} \sin x \,\cdotp\sinh y \cr -\cos x \,\cdotp\sinh y \cr x \end{matrix}\right).
\end{equation}
At $t=0$ we have a {\em catenoid} (see \eqref{eq:catenoid}), and at $t=\pm \pi/2$ 
we have a {\em helicoid} (see \eqref{eq:helicoid}).
Hence, these are conjugate minimal surfaces in $\R^3$. The holomorphic null curve 
\eqref{eq:helicatenoid} is called {\em helicatenoid}.
It is easily verified bthat the given parameterizations of these surfaces
are conformal; of course this also follows from the general theory explained above.
\qed\end{example}

%
%
\smallskip\noindent 
{\bf Weierstrass representation of minimal surfaces in $\R^3$.} 
Let us write $\di X=\frac{\di X}{\di z}dz$. In dimension $n=3$ the Enneper--Weierstrass 
representation formula for a conformal minimal immersion $X=(X_1,X_2,X_3):D\to\R^3$ 
can be written in the more concrete form (see \cite{Osserman1986} or 
\cite[pp.\ 107--108]{AlarconForstnericLopez2021book} for the details): 
\begin{equation}\label{eq:EWR3}
	X(z) = X(z_0) + 2 \Re \int_{z_0}^z
	\left( \frac{1}{2} \Big(\frac{1}{\ggot}-\ggot\Big),
	 \frac{\imath}{2} \Big(\frac{1}{\ggot}+\ggot\Big),1\right) \di X_3,
\end{equation}
where 
\begin{equation}\label{eq:CGauss0}
	\ggot =  \frac{\di X_3}{\di X_1 -\imath \, \di X_2} : D \lra \CP^1
\end{equation}
is a holomorphic map to the Riemann sphere (a meromorphic function on $D$)
called the {\em complex Gauss map} of $X$. 
It is easily seen that the $\C^3$-valued meromorphic $1$-form 
\[
	\Phi=(\phi_1,\phi_2,\phi_3) 
	=  \left( \frac{1}{2} \Big(\frac{1}{\ggot}-\ggot\Big),
	 \frac{\imath}{2} \Big(\frac{1}{\ggot}+\ggot\Big),1\right) \di X_3
\]
in \eqref{eq:EWR3} has no zeros and poles on $M$ if and only if the following two conditions hold:
\begin{itemize}
\item  if at some point $p\in D$ the meromorphic function $\ggot$ has either a zero or a pole of
order $k\in\N$, then $\phi_3=\di X_3$ has a zero at $p$ of the same order $k$, and
\item $\phi_3$ does not have any other zeros (since those would also be zeros of $\Phi$).
\end{itemize}
The complex Gauss map \eqref{eq:CGauss0} corresponds to the classical Gauss map 
$\bN:D\to S^2$ of the immersed minimal surface $X:D\to\R^3$, defined by
\[
	\bN = \frac{X_x \times X_y}{|X_x \times X_y|}, 
\]
provided that we identify the $2$-sphere $S^2\subset \R^3$ with the Riemann sphere $\CP^1$
via the stereographic projection from the point $(0,0,1)\in S^2$. See 
\cite{AlarconForstnericLopez2021book,Osserman1986} for further details. 

One of the most interesting and important features of the complex Gauss map is that the total 
Gaussian curvature $\TC(X)$ (see \eqref{eq:TC}) of a conformal minimal surface $X:D\to\R^3$ equals the 
negative spherical area of the image of the Gauss map $\ggot:D\to\CP^1$ (counted with multiplicities),
where the area of $\CP^1=S^2$ is $4\pi$. Explicitly:
\begin{equation}\label{eq:TCarea}
	\TC(X) = - \Area(\ggot(D)).
\end{equation}
It is a recent result that every holomorphic map $D\to \CP^1$ is the complex Gauss map
of a conformal minimal immersion $X:D\to\R^3$; 
see \cite{AlarconForstnericLopez2019JGEA} or \cite[Theorem 5.4.1]{AlarconForstnericLopez2021book}.
Hence, the total Gaussian curvature of a minimal surface can be any number in $[-\infty,0]$.

%
%
\medskip
\noindent{\bf Minimal surfaces of finite total curvature.} 
Recall from \eqref{eq:TC} that the total curvature of a minimal surface $S\subset \R^3$ 
is the integral $\TC(S)=\int_S K\,\cdotp dA$ of the Gaussian curvature function $K\le 0$ 
with respect to the surface area. If $X:D\to \R^n$ is a minimal immersion,
we have $\TC(X)=\int_D K dA$ where $dA$ is the area measure of the
Riemannian metric $g=X^*ds^2= \sum_{i=1}^n (dX_i)^2$ on $D$ (the first fundamental form of $X$).
The Gaussian curvature function $K:D\to (-\infty,0]$ of the immersion $X$ is determined solely in terms
of the metric $g$ according to Gauss's famous {\em Theorema Egregium}
(the {\em wonderful theorem}). 
An immersed minimal surface $X:D\to\R^3$ is said to have {\em finite total curvature} if 
\[	
	\TC(X)>-\infty.    
\]
The immersed surface $S=X(D)$ is said to be {\em complete} if the $X$-image of any curve in 
$D$ that is not contained in a compact subset of $D$ has infinite Euclidean length in $\R^n$.
This is equivalent to asking that the distance function on $D$ induced by the Riemannian
metric $g=X^*ds^2$ is a complete metric.

If a minimal surface $X:D\to\R^3$ is both complete and of finite total curvature $\TC(X)>-\infty$,
then the domain $D$ must be equal to the complement $\C\setminus \{p_1,\ldots,p_m\}$ 
of finitely many points in $\CP^1$ (or in any compact Riemann surface) 
and the Gauss map $\ggot$ has a pole at each of these punctures (including at $\infty$),
so it extends to a holomorphic map $\ggot:\CP^1\to \CP^1$. 
This is a special case of a classical theorem due to Chern and Osserman \cite{ChernOsserman1967JAM} 
from 1967 which describes complete minimal surfaces in $\R^n$ of finite total curvature.
In such case, the extended Gauss map $\ggot:\CP^1\to \CP^1$ has a well-defined degree 
$\deg(\ggot)\in \Z_+$ which equals the number of points on any fibre $\ggot^{-1}(z)$, $z\in\CP^1$, 
counted with multiplicities. The area of the image of $\ggot$ is then equal to 
$\deg(\ggot)$ times the area of $\CP^1$ (which is $4\pi$), so the formula \eqref{eq:TCarea} 
implies 
\begin{equation}\label{eq:TC4pi}
	\TC(X) = -  4\pi \deg(\ggot).
\end{equation}
That is, {\em if a complete conformal minimal surface $X:D\to \R^3$ has finite total Gaussian curvature, 
then this total curvature is a nonnegative integer multiple of $-4\pi$.} 
The case $\TC(X)=0$ corresponds to planes. Conversely, if the Gauss map of a conformal minimal surface
$X:D=\C\setminus \{p_1,\ldots,p_m\} \to\R^3$ extends to 
a holomorphic map $\CP^1\to\CP^1$ with a pole at each point $p_j$ and
at $\infty$, then $X$ is complete and has finite total curvature.

%
%
\medskip
\noindent{\bf On the Calabi-Yau problem for minimal surfaces.} 
Note that every {\em proper} immersion $X:D\to\R^n$ (i.e., such that every sequence
$p_i\in D$ which diverges to $bD$ and has no limit points inside $D$ is mapped 
to a sequence $X(p_i)\in\R^n$ diverging to $\infty$) is complete in the sense described above.
However, if $D$ is a bounded domain in $\C$ with piecewise smooth boundary,
there also exist conformal minimal immersions $X:D\to\R^n$ for any $n\ge 3$ which are bounded
(i.e., such that the image $X(D)$ lies in a ball of $\R^n$) and complete. Such an immersion must be
highly oscillating at each boundary point of $D$. Nevertheless, it is possible to choose $X$
to extend continuously to $\overline D$ and such that $X(bD)$ is a finite collection
of pairwise disjoint Jordan curves in $\R^n$. If $n\ge 5$ then $X$ can even be chosen a 
topological embedding $\overline D\hra \R^n$. 
The mentioned results belong to the scope of problems around
the famous {\em Calabi-Yau problem for minimal surfaces}; we refer 
to \cite{AlarconForstneric2019JAMS} and
\cite[Chapter 7]{AlarconForstnericLopez2021book} for surveys of this subject.

%
%

\section{A few examples of minimal surfaces}\label{sec:examples}

We conclude by mentioning a few of the simplest examples of minimal surfaces, 
show their conformal parameterizations and the Weierstrass representation.
More precise descriptions of these and other examples, including also their illustrations, 
can be found in \cite[Sect.\ 2.8]{AlarconForstnericLopez2021book} and in many other sources
mentioned in the introduction. Nonorientable minimal surfaces are treated in the
recent publication \cite{AlarconForstnericLopezMAMS}.

%
%
\smallskip
\noindent {\bf The catenoid} is obtained by rotating the {\em catenal curve} 
in $\R^2$ (the graph of the hyperbolic cosine function) around a suitable axis in $\R^3$.
It was described by Leonhard Euler in 1744 and characterized by
Pierre Ossian Bonnet in 1860 as the only rotational minimal surface in $\R^3$,
besides the plane. For example, by rotating the catenal curve 
$\R\ni z\mapsto (\cosh z,0,z)\in\R^3$ around the $z$-axis we obtain 
the catenoid in $\R^3$  given by the implicit equation
\begin{equation}\label{eq:catenoid}
	x^2+y^2=\cosh^2 z.
\end{equation}
Other catenoids are obtained from this model one by rigid motions and dilations.
For example, dilating the coordinates by the factor $c>0$ gives the family
of catenoids
\begin{equation}\label{eq:catenoid2}
	x^2+y^2=c^{-2} \cosh^2 (cz).
\end{equation}
A conformal parameterization of this catenoid is given by the map $X:\R^2 \to \R^3$,
\begin{equation}\label{eq:catenoidparametric}
	X(u,v)= \left(\cos u \,\cdotp \cosh v, \sin u \,\cdotp \cosh v, v\right).
\end{equation}
This is the real part of the holomorphic null curve $Z=X+\imath Y:\C\to\C^3$ given by
%
%
\begin{equation}\label{eq:helicatenoid2}
	Z(\zeta) = (\cos\zeta,\sin\zeta,-\imath \zeta)\in \C^3, \quad\  \zeta=u+\imath v\in\C,
\end{equation}
called {\em helicatenoid} (see Example \ref{ex:helicatenoid}). 
The Enneper--Weierstrass representation of the helicatenoid is 
\begin{eqnarray*}
	Z(\zeta) &=& (1,0,0)+ \int_0^\zeta (-\sin \xi,\cos\xi,-\imath)\, d\xi \\
	&=& (1,0,0)+ \int_0^\zeta \left(\frac12 \left(\frac{1}{\E^{\imath \xi}} - \E^{\imath \xi}\right),
	 \frac{\imath}{2} \left(\frac{1}{\E^{\imath \xi}} + \E^{\imath \xi}\right), 1 \right) (-\imath)d\xi
\end{eqnarray*}
with $\zeta\in\C$. Comparing with \eqref{eq:EWR3} we see that
the Gauss map of the helicatenoid, and hence of all its associated
minimal surfaces \eqref{eq:helicatenoidassociated}, is $\ggot(\zeta) = \E^{\imath \zeta}$.

The parameterization of the catenoid given by \eqref{eq:catenoidparametric} 
is $2\pi$-periodic in the $u$ variable,
hence infinitely sheeted. By introducing the variable $w=\E^{\imath \zeta} = \E^{-v+\imath u} \in \C^*$,
we pass to the quotient $\C/2\pi\,\Z\cong\C^*$ and obtain a single sheeted parameterization
$F:\C^* \to \R^3$ of the same catenoid having the Weierstrass representation
\begin{equation}\label{eq:WRcatenoid}
	F(w) = (1,0,0) - \Re \int_1^w \left(\frac12 \Big(\frac1{\eta}-\eta\Big),
	\frac{\imath}{2} \Big(\frac1{\eta}+\eta\Big),1 \right)\frac{d\eta}{\eta}.
\end{equation}
(We introduced the variable $\eta=\E^{\imath\xi}$ into the integral for $Z(\zeta)$.
This gives the same parametric expression $X(u,v)$ \eqref{eq:catenoidparametric}
in terms of the local conformal coordinates $(u,v)=(\mathrm{Arg}(w),-\log|w|)$.)
From the formula \eqref{eq:EWR3} we see that the
complex Gauss map \eqref{eq:CGauss0} of the catenoid parameterized by \eqref{eq:WRcatenoid} is
\[
	\ggot(w)=w, \quad\ w\in\C^*,
\]
so it extends to the identity map $\CP^1\to \CP^1$ of degree $1$.
It follows from \eqref{eq:TC4pi} that the catenoid has total Gaussian curvature equal to $-4\pi$.
In fact, the catenoid is the only surface in the family \eqref{eq:helicatenoid2}
which factors through $\C^*$, has meromorphic Weierstrass data and is of finite total curvature;
all other surfaces in the family \eqref{eq:helicatenoid2} with $t\notin \pi\Z$ are transcendental and
of infinite total curvature.

The catenoid is the most paradigmatic example in the theory of minimal surfaces, and a list
of its major properties can be found in \cite[Subsect.\ 2.8.1]{AlarconForstnericLopez2021book}.

%
%
\begin{example}\label{ex:nonuniqueness}
The family of catenoids \eqref{eq:catenoid2} shows that 
the Plateau boundary value problem can have more than one solution.
We wish to find a piece of a catenoid from \eqref{eq:catenoid2} whose boundary 
consists of the circles $x^2+y^2=r^2$ in the planes $z=1$ and $z=-1$.
The equation for the boundary values is $r^2 = c^{-2}\cosh^2(c)$ or
\[
	 \cosh c = r c,\quad c>0.
\]
There is a number $r_0 \approx 1.50887954\cdots$ such that this equation
has two solutions for $r>r_0$, one double solution for $r=r_0$, and no solutions for  $r<r_0$.
Hence, for values $r>r_0$ we have two catenoids satisfying these boundary conditions.
As the radius $r$ of the circles decreases, the catenoid brakes at the threshold value $r_0$ 
and there is no catenoid (and in fact no minimal surface) connecting this pair of circles for $r<r_0$. 
\qed\end{example}

%
%
\medskip
\noindent {\bf The helicoid} was described by Leonard Euler in 1774 and
Jean Baptiste Meusnier in 1776. Geometrically, it is generated by rotating a line in a
plane of $\R^3$ and simultaneously displacing it in the perpendicular direction which is the axis of rotation.
Therefore, it is invariant under a one parameter family of screw motions around the axis of
rotation, and consequently it is foliated by helices; hence its name.

Let $Z:\C\to\C^3$ be the helicatenoid \eqref{eq:helicatenoid2}.
From \eqref{eq:helicatenoidassociated} we obtain the following
conformal  parameterization of the helicoid $Y=-\Im Z = \Re(\imath Z):\R^2\to\R^3$:
\begin{equation}\label{eq:helicoid}
	 Y(u,v) = (\sin u \,\cdotp\sinh v, -\cos u \,\cdotp\sinh v, u).
\end{equation}
Its Weierstrass representation is
\[
	Y(\zeta)= \Re \int_0^\zeta
	\left(\frac12 \left(\frac{1}{\E^{\imath \xi}} - \E^{\imath \xi}\right),
	 \frac{\imath}{2} \left(\frac{1}{\E^{\imath \xi}} + \E^{\imath \xi}\right), 1 \right) d\xi,
	 \quad\ \zeta\in\C.
\]
Since its complex Gauss map $\ggot(\zeta)=\E^{\imath \zeta}$ (see \eqref{eq:EWR3}) is transcendental,
the total curvature of the helicoid equals $-\infty$.

%
%
\medskip
\noindent {\bf Enneper's surface} was discovered by Alfred Enneper in 1868.
It is one of the two most basic minimal surfaces in $\R^3$ from the point of view of the
Enneper--Weierstrass representation, the other one being the catenoid.
Its parametric equations  are
\[	
	X(u,v)=
 	\left( \frac{u}{3}\big(3 (1 + v^2)-u^2\big), \frac{v}{3} \big(v^2 -3(1 + u^2)\big), u^2 - v^2\right),
\]
and the Weierstrass representation is given by
\[
	X:\C\to \R^3, \qquad 
	X(\zeta) =\Re \int_0^\zeta \left( 1-\xi^2,\imath(1+\xi^2), 2\xi \right) d\xi.
\]
Hence, the complex Gauss map is $\ggot(\zeta)=\zeta$, and \eqref{eq:TC4pi} shows that
the total Gaussian curvature equals $\TC(X)=-4\pi$. 
It turns out that Enneper's surface is conjugate to itself.
Besides the catenoid, Enneper's surface is the only complete minimal surface
in $\R^3$ of total Gaussian curvature $-4\pi$ (see Osserman \cite{Osserman1986}).

%
%
\medskip
\noindent {\bf Meeks's minimal M{\"o}bius Strip} was discovered by William H.\ Meeks \cite{Meeks1981} in
1981. It was the first known example of a nonorientable properly immersed minimal surface in $\R^3$.
Its orientable double cover is parameterized by the harmonic map $X:\C^* \to \R^3$ with 
the following Weierstrass representation formula (see Meeks \cite[Theorem 2]{Meeks1981}):
\[
	X(\zeta)=\Re\int_1^\zeta \left( \frac{z-1}{z^2 (z+1)}-\frac{z^2 (z+1)}{z-1},
	\imath \left(\frac{(z+1) z^2}{z-1}+\frac{z-1}{z^2 (z+1)}\right),2 \right) \frac{\imath (z^2-1)}{2 z^2} dz,
\]
where $\Igot:\C^*\to \C^*$, $\Igot(z)=-1/\bar z$, is the associated antiholomorphic deck transformation.
The Gauss map $\ggot : \CP^1\to \CP^1$ of $X$ equals 
\[
	\ggot(z)=\frac{(z+1) z^2}{z-1}.
\] 
Clearly it has degree $3$, so $X$ has total Gaussian curvature $-12\pi$ by \eqref{eq:TC4pi}. 
Since the map $X$ covers the M\"obius strip exactly twice, the latter has total curvature $-6\pi$.
As shown by Meeks in \cite[Sect.\ 4]{Meeks1981}, this is the only complete nonorientable immersed
minimal surface in $\R^3$ with the absolute total Gaussian curvature smaller than $8\pi$.

%
%
\medskip
\noindent {\bf The Alarc\'on--Forstneri\v c--L\'{o}pez M{\"o}bius Strip} is the first known 
example of a {\em properly embedded} nonorientable minimal surface
in $\R^4$. It was found in 2017 and is described in 
\cite[Example 6.1]{AlarconForstnericLopezMAMS}. 
No such example exists in $\R^3$ according to Meeks \cite[Corollary 2]{Meeks1981}.

Let $\Igot : \C^*\to\C^*$ be the fixed-point-free antiholomorphic involution
$\Igot(\zeta)= -1/{\bar\zeta}$. The harmonic map $X :  \C^*\to\R^4$ defined by
\[
	X(\zeta)=\Re\left( \imath\Big(\zeta+\frac1{\zeta}\Big) \,,\, \zeta-\frac1{\zeta}  \,,\,
	\frac{\imath}2 \Big(\zeta^2-\frac1{\zeta^2}\Big)  \,,\, \frac12 \Big(\zeta^2+\frac1{\zeta^2}\Big) \right)
\]
is an $\Igot$-invariant proper conformal minimal immersion such that 
$X(\zeta_1)=X(\zeta_2)$ if and only if $\zeta_1=\zeta_2$ or $\zeta_1=\Igot(\zeta_2)$.
Hence, the image surface $S=X(\C^*)\subset \R^4$ is a properly embedded minimal
M\"obius strip in $\R^4$. It has total Gaussian curvature $-4\pi$.

%
%
%
%
\subsection*{Acknowledgements}
Research is supported by the research program P1-0291 and the grant J1-9104
from ARRS, Republic of Slovenia. I wish to thank Finnur L\'arusson for his 
remarks and suggestions which led to improved presentation.




\vspace*{5mm}
\noindent Franc Forstneri\v c

\noindent Faculty of Mathematics and Physics, University of Ljubljana, Jadranska 19, SI--1000 Ljubljana, Slovenia

\noindent 
Institute of Mathematics, Physics and Mechanics, Jadranska 19, SI--1000 Ljubljana, Slovenia.

\noindent e-mail: {\tt franc.forstneric@fmf.uni-lj.si}

\end{document}